\documentclass[10pt,double,reqno]{amsart}

\usepackage{amsmath,amssymb,amscd,color}

\usepackage{graphicx}
\usepackage{epsfig}
\usepackage{amsthm}

\usepackage{enumerate}
\usepackage{amsbsy}
\usepackage{amsfonts}
\newtheorem{theo}{Theorem}[section]

\newtheorem{lemm}[theo]{Lemma}
\newtheorem{rem}[theo]{Remark}

\newcommand{\al}{\alpha}

\newcommand{\la}{\lambda}

\newcommand{\ep}{\epsilon }

\newcommand{\R}{{\mathbb R}^d}

\newcommand{\ri}{\rightarrow}

\begin{document}
\baselineskip=18pt

\title[]{Transition densities of one-dimensional L$\acute{\mbox{E}}$vy processes}

\author{Tongkeun Chang}
\address{Department of Mathematics, Yonsei University \\
Seoul, 136-701, South Korea}
\email{chang7357@yonsei.ac.kr}

\begin{abstract}
In this paper, we study the existence of the transition densities of one-dimensional L$\acute{\mbox{e}}$vy processes.
Compared with past results, our  results contain the  L$\acute{\mbox{e}}$vy processes whose L$\acute{\mbox{e}}$vy symbols have logarithm behavior at infinity.
Our results contain the L$\acute{\mbox{e}}$vy symbol induced by  the following Laplace exponent
\begin{align}\label{psiex}
\psi(\xi): = \underbrace{(\ln(1 + \ln ( 1 + \ln( \cdots \ln(1 + |\xi|)))))^\ep}_{n  \,\, times}, \qquad 0 < \ep \leq 1, \quad   2\leq n.
\end{align}
We also show that $\psi$ defined by  \eqref{psiex} is  a L$\acute{\mbox{e}}$vy symbol with transition density.\\



\end{abstract}

\maketitle

\section{Introduction}
\setcounter{equation}{0}
In this paper, we study the existence of the transition densities of one-dimensional L$\acute{\mbox{e}}$vy processes.
Let $\{P_t, \, t > 0\}$ be a distribution of the  L$\acute{\mbox{e}}$vy process $X_t$ whose Fourier transform is  ${\mathcal F}(P_t)(\xi) = \int_{{\mathbb R}} e^{i\xi\cdot y} P_t(dy) = e^{-t\eta(\xi)}$, where the L$\acute{\mbox{e}}$vy symbol $\eta$ has the form
\begin{align*}
\eta(\xi) = -i b \xi  + \frac12 A \xi^2 - \int_{{\mathbb R}} \Big(e^{i\xi \cdot y} -1 - i \xi  y I_{(-1,1)} (y) \Big) \nu (dy)
\end{align*}
with constant $b$, non-negative constant $A$ and  L$\acute{\mbox{e}}$vy measure $\nu$.
The L$\acute{\mbox{e}}$vy measure $\nu$ satisfies
\begin{align*}
\int_{{\mathbb R}} \min( 1, |y|^2 ) \nu (dy) < \infty.
\end{align*}

The transition densities of the L$\acute{\mbox{e}}$vy processes are important tools in theoretical probability, physics and finance.
So, the existence of the transition densities and the asymptotic property of the L$\acute{\mbox{e}}$vy processes, in particular $\al$-stable processes,  have been studied  by many mathematicians (see \cite{BB,BJ,BS,CK,CK2,D,G,GH,PT,S1, S2, W,Za,Zo}). In those studies,  transition densities have the desirable properties of continuity and boundedness,  because the L$\acute{\mbox{e}}$vy symbols demonstrate asymptotic,  polynomial behavior at infinity.

In the present paper, we assume that the L$\acute{\mbox{e}}$vy symbol $\eta$ has  logarithmic behavior at infinity (see \eqref{psiex}). To demonstrate the existence of transition density, we use Fourier-analytic methods. In general, the differentiability of  ${\mathcal F} (f)$  is related to the behavior of $f$ at infinity and the behavior at infinity of  ${\mathcal F} (f)$ is related to  the differentiability of $f$.
Therefore, transition densitie, do not have good regularity  (for example, continuity and boundedness), since we assume that the L$\acute{\mbox{e}}$vy symbol $\eta$ has logarithmic behavior at infinity. Instead, to determine the behavior of transition density at infinity, we will assume that $\eta$ lies in $C^2({\mathbb R} \setminus \{0\})$.
Our main results are as follows;

\begin{theo}\label{Emaintheo}
Let  $\eta (\xi)   =\eta_1(\xi) +i\eta_2 (\xi)$ be a L$\acute{\mbox{e}}$vy symbol of the one-dimensional L$\acute{\mbox{e}}$vy process $X_t$. Suppose that $\eta_1$ and $\eta_2$ satisfy the following assumptions: there are a $0 < \ep \leq  1$  and $\al_\ep > 0$ such that, for $\xi \in {\mathbb R}$,
\begin{align}\label{mainassumption}
\notag \eta_1(\xi) & \geq \al_\ep s_n^{\ep}(\xi),\\
 \notag |\eta_2(\xi)| & \leq   c \left\{\begin{array}{l} \vspace{2mm}
|\xi|^\ep, \quad |\xi| \leq 1,\\
s_n^{\ep-1} (\xi)  r^{-1}_{n-1}(\xi), \quad |\xi| \geq 1,
\end{array}
\right.\\
\notag |\eta^{'}_1 (\xi) |, \, |\eta_2^{'}(\xi) | & \leq  c \left\{\begin{array}{l} \vspace{2mm}
   |\xi|^{\ep -1}, \quad |\xi| \leq 1,\\
 s_n^{\ep -1}(\xi) r^{-1}_{n-1}(\xi) (1 + |\xi|)^{-1}, \quad |\xi| \geq 1,
  \end{array}
  \right.\\
  | \eta^{''}_1 (\xi) |, \, |\eta^{''}_2(\xi) | & \leq  c
  \left\{\begin{array}{l}\vspace{2mm}
   |\xi|^{\ep -2}, \quad |\xi| \leq 1,\\
  s_n^{\ep -1}(\xi) r^{-1}_{n-1}(\xi)    (1 + |\xi|)^{-2}, \quad |\xi| \geq 1,
 \end{array}
 \right.
\end{align}
where
\begin{align*}
s_1 (\xi) &= \ln (|\xi|+1),\\
s_n(\xi) &: =   \ln (1 + s_{n-1}(\xi)) , \quad n \geq 2,\\
r_n(\xi) &:  = s_{n }(\xi) s_{n-1 }(\xi) \cdots s_1(\xi) ( = s_n(\xi) r_{n-1} (\xi)).
\end{align*}
Then, $X_t$ has a transition density $p_t$ such that for all $0 <t < \infty$,   $p_t$ satisfies that,
if $|x| \leq 1$, then
\begin{align*}
p_t(x) \leq c (t)
      \frac1{|x|}      e^{-\al_\ep t    s^\ep_n(\frac1x)}          s^{\ep-1}_n(\frac1x) r_{n-1}^{-1}(\frac1x)
\end{align*}
and, if $|x| \geq 1$, then
\begin{align*}
p_t(x) \leq c (t) \left\{ \begin{array}{ll} \vspace{2mm}
     |x|^{-1 -\ep}, &  0 < \ep <1,\\
    |x|^{-2} \ln(1+ |x|),&\ep =1,
\end{array}
\right.
\end{align*}
where $c(t)$ is positive constant dependent on $t$ such that $c (t) \leq c t$ for $t \leq 1$. Moreover, if $| \eta^{''}_1 (\xi) |$ is integrable in the interval $(0, 1)$, then
for $|x| \geq 1$,
\begin{align*}
p_t(x) \leq c(t) |x|^{-2},\qquad \ep =1.
\end{align*}
\end{theo}

\begin{theo}\label{eta2=0}
Let   $\eta $ be a symmetric  real-valued function satisfying the  assumption of Theorem \ref{Emaintheo} such that
\begin{align}\label{1001eta2=0}
\notag  \eta(\xi)   & \leq \al_0 s^{\ep}_n(\xi),\\
   - \eta^{''}_1 (\xi)  & \geq  c \left\{\begin{array}{l}\vspace{2mm}
    s_{n-1}^{\ep -1}(\xi) r^{-1}_{n-1}(\xi) (1 + \xi)^{-2}, \quad |\xi| \geq 1,\\
  |\xi|^{\ep -2}, \quad |\xi| \leq 1.
 \end{array}
 \right.
\end{align}
Then, the transition density $p_t$  of $X_t$ satisfies  that if $|x| \leq 1$, then
\begin{align*}
p_t(x) \geq  c (t)
      \frac1{|x|}      e^{-\al_0 t    s^\ep_n(\frac1x)}          s^{\ep-1}_n(\frac1x) r_{n-1}^{-1}(\frac1x)
 \end{align*}
and, if $|x| \geq 1$, then
\begin{align*}
p_t(x) \geq c(t)\left\{\begin{array}{l} \vspace{2mm}
         \frac{1}{|x|^{2-\ep} },\quad  0 < \ep <1,\\
        \frac{1}{|x|^2 }   \ln (1 + |x|), \quad \ep =1.
     \end{array}
     \right.
\end{align*}
Moreover, if $| \eta^{''}_1 (\xi) |$ is integrable in the interval $(0, 1)$, then,
for $|x| \geq 1$,
\begin{align*}
p_t(x) \geq c(t) |x|^{-2},\qquad \ep =1.
\end{align*}

\end{theo}

The typical examples of the one-dimensional L$\acute{\mbox{e}}$vy processes are subordinators.
Let the L$\acute{\mbox{e}}$vy process $S_t$ be a subordinator, that is,
\begin{align*}
S_t \geq 0 \quad \mbox{a.s. for each} \quad t> 0,\\
S_{t_1} \leq S_{t_2} \quad \mbox{a.s. whenever} \quad t_1 < t_2.
\end{align*}
The L$\acute{\mbox{e}}$vy symbol $\eta$ of subordinator $S_t$ has the form
\begin{align}\label{sub-characteristic}
\eta(\xi) = ib\xi - \int_0^\infty (e^{i\xi y} -1) \la (dy),
\end{align}
where $b \geq 0$ and the L$\acute{\mbox{e}}$vy measure $\la$ satisfies the additional requirements
\begin{align}\label{sub-assumptions}
\la(-\infty, 0) =0, \quad \int_0^\infty \min(y,  1) \la (dy) < \infty.
\end{align}
Conversely, any mapping  $\psi:{\mathbb R} \ri {\mathbb C}$ of the form \eqref{sub-characteristic} is the L$\acute{\mbox{e}}$vy symbol of a subordinator.

Now, if $S_t$ is subordinator, then, for each $t \geq 0$, the map $ f(\xi) =  Ee^{i\xi S_t} = e^{-t\eta(\xi)}$ can
  be analytically continued to the region $\{ v + i\xi\, | \, v \in {\mathbb R},  \, \xi> 0\}$.
Let $F(z)   =  e^{-t\eta(z)} $ be an analytical extension of $e^{-t\eta(\xi)}$ over $\{ v + i\xi\, | \, v \in {\mathbb R},  \, \xi> 0\}$. Then, we get
\begin{align*}
F(i\xi) = E e^{-\xi S_t} = e^{-t \eta(i\xi)}: = e^{-t\psi(\xi)}, \qquad \xi \geq 0,
\end{align*}
where
\begin{align}\label{levyexponent}
\psi(\xi) =  \eta(i \xi) = - b\xi - \int_0^\infty ( e^{-\xi y} -1) \la(dy)
\end{align}
for each $\xi > 0$. This is much more useful for theoretical and practical applications than is the L$\acute{\mbox{e}}$vy symbol.
The function $\psi$ is usually called the Laplace exponent of the subordinator.

The Laplace exponents of subordinators are characterized by the Bernstein functions. We say that the continuous function $\psi: [0, \infty) \ri [0, \infty)$ is a  Bernstein function
if   $(-1)^k\psi^{(k)} \leq 0$ for all $k \in {\mathbb N}$.  If $\psi$ is the Laplace exponent of the subordinator, then, from
 \eqref{levyexponent}, $\psi$ is a Bernstein function.

Conversely, if $\psi$ is a Bernstein function, then there are a non-negative real number $b$ and  a measure $\la$ defined in ${\mathbb R} $  satisfying \eqref{sub-assumptions} such that \eqref{levyexponent}  holds. Moreover, there is a subordinator $S_t$ such that
\begin{align}\label{bernstein}
\eta(\xi) = \psi(-i\xi)
\end{align}
is  the L$\acute{\mbox{e}}$vy symbol of $S_t$.

This paper is organized as follows. In section \ref{Preli},  we introduce two lemmas  to prove  Theorem \ref{Emaintheo}  and Theorem \ref{eta2=0}. In section \ref{sec7} and section \ref{sec5}, we prove Theorem \ref{Emaintheo} and  Theorem \ref{eta2=0}, respectively. In section \ref{examples}, we introduce  examples satisfying Theorem \ref{Emaintheo} and  Theorem \ref{eta2=0}.

In this paper, we  denote by $c$  various
generic positive constants and by $c(*,\cdots,*)$   the constants depending only on the quantities appearing in the
parenthesis.

\section{ Main Lemmas }\label{Preli}
\setcounter{equation}{0}

In this section, we introduce the two lemmas for the proofs of Theorem \ref{Emaintheo} and  Theorem \ref{eta2=0}.
The first lemma is as follows.
\begin{lemm}\label{mainlemma}
Suppose that,  for $\xi \in {\mathbb R} \setminus \{0\}$,
\begin{align}\label{1001}
\notag &\eta_1(\xi) \geq g_1(|\xi |), \quad |\eta_2(\xi)| \leq g_2(|\xi|),\\
\notag &|\eta_1^{'}(\xi)|, \,\, |\eta_2^{'}(\xi)| \leq g_3(|\xi|),\\
& |\eta_1^{''}(\xi)|, \,\, |\eta_2^{''}(\xi)| \leq g_4( |\xi| ).
\end{align}
Let $f_1(\xi) = e^{-t\eta_1(\xi)} \cos \, t\eta_2(\xi), \,f_2(\xi) = e^{-t\eta_1(\xi)} \sin \, t\eta_2(\xi)$  and
$$
G(\xi) :=  e^{-tg_1 (\xi)}\big(t^2 g_3(\xi)^2 + t g_4(\xi) \big).
$$
Suppose that $G$ is decreasing in $(0, \infty)$.
Then, for $x \in {\mathbb R}$ and $t \in (0, \infty)$,
\begin{align*}
|\int_{{\mathbb R}} f_1(\xi) \cos x \xi d\xi|
& \leq c \Big( \frac1{|x|^2} \int_{\frac{2\pi}{|x|}}^{\infty} G(y) dy   + \frac1{|x|^3}      G(\frac1{x})   +  \int_0^{\frac{\pi}{2|x|}} G(y)   y^2 dy    \Big),\\
|\int_{{\mathbb R}} f_2(\xi) \sin x \xi d\xi |
& \leq c \Big(  \frac1{|x|^2} \int_{\frac{2\pi}{|x|}}^{\infty} G(y) dy   +  \frac{t}{|x|} \int_0^{\frac12 \pi}  e^{-t g_1( |\frac{y}{x}|)}  g_2( |\frac{y}{x}|)dy     \Big).
\end{align*}

\end{lemm}

\begin{proof}
Note that
\begin{align*}
f_1^{''}(\xi) &= e^{-t\eta_1(\xi)} \Big( (-t\eta^{'}_1 (\xi) )^2 \cos t\eta_2(\xi)   +  (-t\eta^{''}_1 (\xi) ) \cos t\eta_2(\xi)
+    (-t\eta^{'}_1 (\xi) )(-t\eta^{'}_2 (\xi) ) \sin t\eta_2(\xi)\\
&  \qquad +     (-t\eta^{'}_1 (\xi) ) (-t\eta^{'}_2 (\xi) ) \sin t\eta_2(\xi)   +   (-t\eta^{''}_2 (\xi) ) \sin t\eta_2(\xi)
     -(t\eta^{'}_2 (\xi) )^2 \cos t\eta_2(\xi) \Big),\\
f_2^{''}(\xi) &= e^{-t\eta_1(\xi)}  \Big( (-t\eta^{'}_1 (\xi) )^2 \sin t\eta_2(\xi)   +  (-t\eta^{''}_1 (\xi) ) \sin t\eta_2(\xi)
+    (-t\eta^{'}_1 (\xi) )(t\eta^{'}_2 (\xi) ) \cos t\eta_2(\xi)\\
& \qquad +    (-t\eta^{'}_1 (\xi) ) (t\eta^{'}_2 (\xi) ) \cos t\eta_2(\xi)   +   (t\eta^{''}_2 (\xi) ) \cos t\eta_2(\xi)
-   (t\eta^{'}_2 (\xi) )^2 \sin t\eta_2(\xi) \Big).
\end{align*}
From the assumptions \eqref{1001} of $\eta_1$  and $\eta_2$, we get
\begin{align*}
| f^{''}_1(\xi)|, \,\, | f^{''}_2 (\xi)| & \leq ce^{-t\eta_1(\xi)} \Big( (-t\eta^{'}_1 (\xi) )^2   + |-t\eta^{''}_1 (\xi) |  + |-t\eta^{''}_2 (\xi) |
   + (-t\eta^{'}_2 (\xi) )^2 \Big)\\
 & \leq   c e^{-tg_1 (|\xi|)}\big(t^2 g_3(|\xi|)^2 + t g_4(|\xi|)  \big)\\
 & \leq  c G(|\xi|).
\end{align*}
Using the change of variables, we have
\begin{align}\label{f-1}
\notag  \int_{{\mathbb R}}  f_1(\xi)   \cos (x\xi) d\xi
&= \frac1{|x|}\int_{{\mathbb R}}f_1(\frac{\xi}x)  \cos \xi  d\xi\\
& = \frac1{|x|} \sum_{-\infty<  k < \infty} I^1_k(x),
\end{align}
where
\begin{align}\label{I-1}
\notag I^1_k(x)& =\int_{2\pi k}^{2\pi (k +1)}f_1(\frac{\xi}x)  \cos \xi d\xi\\
\notag   & = \int_{2\pi k}^{2\pi k + \frac12 \pi} f_1(\frac{\xi}{x}) \cos \xi d\xi + \int_{2\pi k +\frac{\pi}2}^{2\pi k + \pi} f_1(\frac{\xi}{x}) \cos \xi d\xi\\
 \notag   & \qquad  + \int_{2\pi k + \pi}^{2\pi k + \frac32 \pi  } f_1(\frac{\xi}{x}) \cos \xi d\xi + \int_{2\pi k + \frac32 \pi }^{2\pi k +  2\pi} f_1(\frac{\xi}{x}) \cos \xi d\xi\\
 \notag  & = \int_0^{\frac12 \pi }
    f_1(\frac{2 \pi k+ \xi}{x}) \cos (2 \pi k + \xi)  d\xi
   + \int_0^{\frac12\pi} f_1(\frac{ 2 \pi k + \pi -\xi}{x}) \cos (2 \pi k+ \pi -\xi) d\xi\\
 \notag    & \qquad + \int_0^{\frac12 \pi}  f_1(\frac{ 2  \pi k + \pi + \xi}{x}) \cos  (  2  \pi k + \pi + \xi) d\xi \\
\notag     & \qquad + \int_0^{\frac12 \pi}
       f_1(\frac{2  \pi k+ 2\pi - \xi}{x}) \cos ( 2  \pi  k + 2\pi - \xi) d\xi\\
 & = \int_0^{\frac12 \pi }
     \Big(f_1(\frac{2 \pi k + \xi}{x}) - f_1(\frac{ 2 \pi k + \pi -\xi}{x}) - f_1(\frac{ 2  \pi k + \pi + \xi}{x})
      + f_1(\frac{2  \pi k + 2\pi - \xi}{x}) \Big) \cos  \xi d\xi.
\end{align}
Using the mean-value theorem, there are $\ep_1 \in  (\xi, \pi -\xi), \,\, \ep_2 \in (\pi + \xi, 2\pi -\xi)$ such that

\begin{align*}
f_1(\frac{2 \pi k + \xi}{x}) - f_1(\frac{ 2 \pi k + \pi -\xi}{x}) = -\frac1x f_1^{'}(\frac{2 \pi k+ \ep_1}{x})(\pi -2\xi),\\
f_1(\frac{2 \pi k + \pi + \xi}{x}) - f_1(\frac{ 2 \pi k + 2\pi -\xi}{x}) = -\frac1x f_1^{'}(\frac{2 \pi k + \ep_2}{x})(\pi -2\xi).
\end{align*}
We use the mean-value theorem again such that
\begin{align*}
-\frac1x f_1^{'}(\frac{2 \pi k + \ep_1}{x})(\pi -2\xi) +  \frac1x f_1^{'}(\frac{2 \pi k + \ep_2}{x})(\pi -2\xi) = \frac{1}{x^2} (\pi -2\xi)(\ep_2 - \ep_1)f_1^{''} (\frac{2k \pi + \ep_3}{x}),
\end{align*}
 where $\ep_3$ lies  between $\ep_1$ and $\ep_2$. Hence, we obtain
\begin{align}\label{kneq0}
I^1_k(x)
 & =  \frac{1}{|x|^2} \int_0^{\frac12 \pi }  (\pi -2\xi)(\ep_2 - \ep_1) f_1^{''} (\frac{2 \pi k + \ep_3}{x}) \cos \xi d\xi.
\end{align}
Since $f_1^{''}$ is dominated by the positive and  decreasing function $G$,    we have
\begin{align}\label{I-1-k}
 \notag |I^1_k(x)|
 &  \leq  c \frac{1}{|x|^2} \int_0^{\frac12 \pi}   G(| \frac{2 \pi k + \ep_3|}{|x|}) d\xi\\
 &  \leq  c \left\{\begin{array}{l}\vspace{2mm}
  \frac{1}{|x|^2}  G(\frac{2 \pi k  }{| x|}),\quad k \geq 1,\\
                       \frac{1}{|x|^2}  G(\frac{2 \pi | k+1|  }{| x|}),\quad k \leq -2.
 \end{array}
         \right.
\end{align}
Since $G:(0, \infty) \ri (0, \infty)$ is a decreasing function, we have
\begin{align}\label{k>1}
\notag\sum_{2 \leq |k| } |I^1_k(x)|
 & \leq c  \frac{1}{x^2}  \Big( \sum_{2 \leq k}   G( \frac{2 \pi| k|}{|x|}) + \sum_{k \leq -2}   G( \frac{2 \pi| k +1|}{|x|})   \Big)\\
\notag & \leq c  \frac{1}{x^2}  \sum_{2 \leq k <\infty}   G( \frac{2 \pi k}{|x|})\\
\notag & \leq c \frac1{x^2} \int_1^{\infty}  G( \frac{2 \pi y}{|x|}) dy\\
 & = c \frac{1}{2\pi |x|} \int_{\frac{2\pi}{|x|}}^{\infty} G(y) dy.
\end{align}

In the case $k =0$, from  \eqref{I-1}, we have
\begin{align*}
 \frac1x \int_0^{2 \pi} f_1(\frac{\xi}{x}) \cos \xi d\xi
& =  \frac1x \int_0^{\frac{\pi}2}
     \Big(f_1(\frac{\xi}{x}) - f_1(\frac{  \pi -\xi}{x}) - f_1(\frac{  \pi + \xi}{x})
      + f_1(\frac{2\pi - \xi}{x}) \Big) \cos  \xi d\xi\\
& =  \frac1x \int_0^{\frac{\pi}2} \Big( - \int_{\frac{\xi}x}^{\frac{\pi -\xi}x} f_1^{'}(y) dy
             +  \int_{\frac{\pi + \xi}x}^{\frac{2\pi -\xi}x} f_1^{'}(y) dy \Big) \cos \xi d\xi\\
& =  \frac1x \int_0^{\frac{\pi}2}  \int_{\frac{\xi}x}^{\frac{\pi -\xi}x} ( -f_1^{'}(y) +  f_1^{'}( \frac{\pi}x +y) ) dy
             \cos \xi d\xi\\
& =  \frac1x \int_0^{\frac{\pi}2}  \int_{\frac{\xi}x}^{\frac{\pi -\xi}x} \int_y^{\frac{\pi}x + y} f_1^{''}(z)dz dy
             \cos \xi d\xi.
\end{align*}
Hence, using  Fubini's theorem, for $x > 0$, we have
\begin{align*}
|\frac1{x} \int_0^{2 \pi} f_1(\frac{\xi}{x}) \cos \, \xi  d\xi |  &\leq   \frac1x \int_0^{\frac{\pi}2}  \int_{\frac{\xi}x}^{\frac{\pi -\xi}x} \int_y^{\frac{\pi}x + y}G(z) dzdyd\xi \\
     &=    \frac1x \int_0^{\frac{\pi}2}  \int_{\frac{\xi}x}^{\frac{\pi -\xi}x}G(z) \int_{\frac{\xi}x}^{z} dy dz d\xi
       + \frac1x  \int_0^{\frac{\pi}2}  \int_{\frac{\pi -\xi}x}^{\frac{\pi +\xi}x}G(z) \int_{\frac{\xi}x}^{\frac{\pi -\xi}x} dy dz d\xi\\
     & \quad     + \frac1x  \int_0^{\frac{\pi}2}  \int_{\frac{\pi +\xi}x}^{\frac{2\pi -\xi}x}G(z)  \int_{z-\frac{\pi}x}^{\frac{\pi - \xi}x} dy dz d\xi \\
     &=   \frac1x  \int_0^{\frac{\pi}2}  \int_{\frac{\xi}x}^{\frac{\pi -\xi}x}G(z) (z -\frac{\xi}x)  dz d\xi
     + \frac1x  \int_0^{\frac{\pi}2}  \int_{\frac{\pi -\xi}x}^{\frac{\pi +\xi}x}G(z)\frac{\pi -2\xi}{x} dz d\xi\\
     &\quad    +  \frac1x  \int_0^{\frac{\pi}2}  \int_{\frac{\pi +\xi}x}^{\frac{2\pi -\xi}x}G(z) ( \frac{2\pi -\xi}x -z) dz d\xi\\
     &: = I_{0,1}^1(x) + I_{0,2}^1(x) + I_{0,3}^1(x).
\end{align*}
Using  Fubini's theorem, we get
\begin{align}\label{I-0-1}
\notag  I_{0,1}^1(x)
& =\frac1x \int_0^{\frac{\pi}{2x}} G(z) \int_0^{xz}(z -\frac{\xi}x)d\xi dz
                    + \frac1x \int_{\frac{\pi}{2x}}^{\frac{\pi}{x}}  G(z) \int_0^{\pi -xz}(z -\frac{\xi}x)d\xi dz\\
\notag & = \frac12 \int_0^{\frac{\pi}{2x}} G(z)   z^2 dz
                    + \frac12 \int_{\frac{\pi}{2x}}^{\frac{\pi}{x}}  G(z) (z^2 -( \frac{\pi}x)^2) dz\\
\notag & \leq      \frac12 \int_0^{\frac{\pi}{2x}} G(z)   z^2 dz
                    + \frac12 G ( \frac{\pi}{2x} ) \int_{\frac{\pi}{2x}}^{\frac{\pi}{x}}   (z^2 -( \frac{\pi}x)^2) dz \\
&  \leq  c\big (     \int_0^{\frac{\pi}{2x}} G(z)   z^2 dz
                    +  \frac1{x^3} G ( \frac1{x} ) \big).
\end{align}

Next, we estimate $I^1_{0,2}$. Using  Fubini's theorem, we have
\begin{align}\label{I-0-2}
\notag I^1_{0,2} & \leq c\frac1{x^2}  \int_0^{\frac{\pi}2}  \int_{\frac{\pi -\xi}x}^{\frac{\pi +\xi}x}G(z) dz d\xi\\
\notag  & = c  \frac1{x^2}    \int_{\frac{\pi }{2x}}^{\frac{3\pi }{2x}}  G(z) dz \\
   & \leq c  \frac1{x^3}      G(\frac1{x}).
\end{align}
Next, we estimate $I^1_{0,3}$. Using  Fubini's theorem, we get
\begin{align}\label{I-0-3}
\notag I^1_{0,3} &= \frac1x \int_{\frac{\pi}x}^{\frac{3\pi}{2x}}  G(z) \int^{xz -\pi}_0 ( \frac{2\pi -\xi}x -z) d\xi dz
            +\frac1x \int_{\frac{3\pi}{2x}}^{\frac{2\pi}x} G(z) \int^{-xz +2\pi}_0 ( \frac{2\pi -\xi}x -z) d\xi dz \\
\notag          &= \frac1x  \int_{\frac{\pi}x}^{\frac{3\pi}{2x}} G(z)   (xz -\pi)(\frac{5\pi}{2x} -\frac{3z}{2}) dz
            + \frac1x \int_{\frac{3\pi}{2x}}^{\frac{2\pi}x} G(z)  (-xz + 2\pi)(\frac{\pi}{x} -\frac12 z)   dz\\
  & \leq c  \frac1{x^3}      G(\frac1x).
\end{align}

Hence, from \eqref{I-0-1}, \eqref{I-0-2} and \eqref{I-0-3}, we have that,  for  $0< x $,
\begin{align}\label{0<x<1}
I^1_0 \leq |\frac{1}x \int_0^{2\pi} f_1(\frac{\xi}x)  \cos \xi d\xi|
& \leq c \big( \frac1{|x|^3}      G(\frac1x)   +  \int_0^{\frac{\pi}{2x}} G(z)   z^2 dz \big).
\end{align}

By similar calculation, \eqref{0<x<1} holds for $  x<0$.
Using the same argument, we obtain
\begin{align}\label{infty<x<0}
I^1_{-1} \leq |\frac{1}x \int_{-2\pi}^0 f_1(\frac{\xi}x)  \cos \xi d\xi|
& \leq c \big( \frac1{|x|^3}      G(\frac1x)   +  \int_0^{\frac{\pi}{2|x|}} G(z)  z^2 dz \big).
\end{align}
Hence, by \eqref{f-1}, \eqref{k>1}, \eqref{0<x<1} and  \eqref{infty<x<0}, we obtain the first inequality of Theorem \ref{mainlemma}.

Next,   we have
\begin{align*}
\int_{{\mathbb R}} f_2(\xi) \sin (x\xi) d\xi
&= \frac1{|x|} \int_{{\mathbb R}} f_2(\frac{\xi}{x}) \sin \xi d\xi\\
& = \frac1{|x|} \int_{-\frac12 \pi}^{\frac12 \pi} f_2(\frac{\xi}{x}) \sin \xi d\xi + \frac1{|x|} \sum_{0 \leq k< \infty} (I^{2,1}_k + I^{2,2}_k ),
\end{align*}
where for $k > 0$,
\begin{align*}
I^{2,1}_k(x) = \int_{2\pi k + \frac12 \pi}^{ 2 \pi (k+1) + \frac12 \pi} f_2(\frac{\xi}{x}) \sin \xi d\xi,\qquad
I^{2,2}_k(x) = \int_{2\pi(- k-1) - \frac12 \pi}^{ 2 \pi (-k) - \frac12 \pi} f_2(\frac{\xi}{x}) \sin \xi d\xi.
\end{align*}
Using  two applications  of mean-value theorem, 
there are $\ep_1 \in (\frac12 \pi +\xi, \frac32 \pi -\xi), \,\, \ep_2 \in (\frac32\pi + \xi, \frac52\pi -\xi)$  and $\ep_3 \in (\ep_1, \ep_2)$
such that
\begin{align*}
I^{2,1}_k(x)& = \int_{2\pi k + \frac12 \pi}^{ 2 \pi (k+1) + \frac12 \pi} f_2(\frac{\xi}{x}) \sin \xi d\xi\\
 & = \int_0^{\frac{\pi}2}
     \Big(f_2(\frac{2 \pi k +\frac12\pi + \xi}{x}) - f_2(\frac{ 2 \pi k + \frac32\pi -\xi}{x}) \\
     & \qquad  - f_2(\frac{ 2  \pi k +\frac32\pi + \xi}{x})
      + f_2(\frac{2  \pi k + \frac52\pi - \xi}{x}) \Big) \sin  \xi d\xi\\
 & = \frac{1}{x^2}\int_0^{\frac{\pi}2}  2\xi (\ep_2 -\ep_1)  f_2^{''} (\frac{2 \pi k + \ep_3}{x}) \sin \xi d\xi.
\end{align*}
Hence, we have
\begin{align}\label{I-2-1}
\notag|I^{2,1}_k(x)| &  \leq  c\frac{1}{|x|^2} \int_0^{\frac{\pi}2}   G(\frac{2 \pi k + \ep_3}{|x|}) d\xi \\
\notag &  \leq  c\frac{1}{|x|^2} \int_0^{\frac{\pi}2}   G(\frac{2 \pi k +\frac12 \pi }{|x|}) d\xi \\
 &  \leq  c \frac{1}{|x|^2}  G(\frac{2 \pi k +\frac12 \pi  }{|x|}).
\end{align}
Using similar calculation to $I^{2,1}_k (x)$, we have
\begin{align}\label{I-2-1}
 |I^{2,2}_k(x)|
 &  \leq  c \frac{1}{|x|^2}  G(\frac{2 \pi |k| +\frac12 \pi  }{|x|}).
\end{align}

Hence, we have
\begin{align*}
\sum_{0 \leq k < \infty} ( |I^{2,1}_k(x)| +  |I^{2,2}_k(x)|)
 & \leq c  \frac{1}{|x|^2}  \sum_{1 \leq k <\infty}   G( \frac{2 \pi k + \frac12 \pi}{|x|})\\
 & \approx \frac1{|x|^2} \int_0^{\infty}  G( \frac{2 \pi y + \frac12 \pi}{|x|}) dy\\
 & = c \frac1{|x|} \int_{\frac{\pi}{2 |x|}}^{\infty} G(y) dy.
\end{align*}

Note that
\begin{align*}
|\frac1x \int_{-\frac12 \pi}^{\frac12 \pi} f_2(\frac{\xi}{x}) \sin \xi d\xi|
& \leq c \frac1{|x|}    \int_{-\frac12 \pi}^{\frac12 \pi}  e^{-t g_1(|\frac{\xi}{x}|)} t g_2(|\frac{\xi}{x}|) \xi \\
& \leq c \frac{t}{|x|}    \int_0^{\frac12 \pi}  e^{-t g_1(|\frac{\xi}{x}|)}  g_2(|\frac{\xi}{x}|) \xi.
\end{align*}

Hence, we obtain the second inequality of Lemma \ref{mainlemma}. We complete the proof of Lemma \ref{mainlemma}.

\end{proof}

The second lemma is as follow.
\begin{lemm}\label{mainlemma2}
Let $0 < \ep \leq 1$ and $n \geq 2$.
\begin{itemize}

\item[(1)]
For $\al > -1$ and $a \geq 1$,
\begin{align*}
  \int_1^{a}  z^{\al}e^{-t\al_\ep s^\ep_n (z)}  s_n^{\ep -1}(z) r^{-1}_{n-1}(z) dz \leq c(\al,t) a^{\al +1}e^{-t\al_\ep s^\ep_n (a)}  s_n^{\ep -1} (a) r^{-1}_{n-1}(a).
\end{align*}

\item[(2)]
Let  $-1 < \al$ and $a_0$ satisfy $(1 + \al)^{-1} (n -\ep +  \ep t \al_\ep  ) \ln (1 +a_0)^{-1}   = \frac12$ (that is, $a_0 = e^{2(1 + \al)^{-1}(n-1 + \al_\ep t)}  - 1$).
For $a \geq a_0$,
\begin{align*}
\int_{a_0}^a z^{\al} e^{-t\al_\ep s^\ep_n (z)}  s_n^{\ep -1} (z)  r^{-1}_{n-1}(z)  dz \leq c(\al)  a^{\al +1}e^{-t \al_\ep s^\ep_n (a)}  s_n^{\ep -1} (a)  r^{-1}_{n-1}(a).
\end{align*}

\item[(3)]
For $\al < -1$ and $a \geq 1$,
\begin{align*}
  \int_{a}^\infty z^{\al} e^{-t  \al_\ep s^\ep_n (z)}   s_n^{\ep -1} (z) r^{-1}_{n-1}(z) dz \leq c (\al) a^{\al +1}e^{-t \al_\ep s^\ep_n (a)}  s_n^{\ep -1} (a) r^{-1}_{n-1}(a).
\end{align*}
\end{itemize}

\end{lemm}

\begin{proof}
To prove (1) of   Lemma \ref{mainlemma2}, it is  sufficient to show
\begin{align*}
\lim_{a \ri \infty} \frac{ \int_1^a z^{\al}e^{-t \al_\ep s^\ep_n (z)}  s_n^{\ep -1} (z)  r^{-1}_{n-1}(z)  dz    }{a^{\al +1}e^{-t \al_\ep s^\ep_n (a)}   s_n^{\ep -1} (a) r^{-1}_{n-1}(a)   } &=(\al +1)^{-1}.
\end{align*}
Using  L' H$ \hat{o}$spital's theorem, we have
\begin{align*}
\lim_{a \ri \infty} \frac{ \int_1^a z^{\al}e^{-t \al_\ep s^\ep_n (z)}  s_n^{\ep -1} (z)  r^{-1}_{n-1}(z)  dz    }{a^{\al +1} e^{-t\al_\ep s^\ep_n (a)}  s_n^{\ep -1} (a) r^{-1}_{n-1}(a)   } =
   \lim_{a \ri \infty} \frac{ a^\al e^{-t \al_\ep s^\ep_n (a) }  s_n^{\ep -1}(a)  r^{-1}_{n-1}(a)   }{ T(a)   },
\end{align*}
where
\begin{align*}
T(a)& = \frac{d}{da}\big(a^{\al +1} e^{-t\al_\ep s^\ep_n (a)}  s_n^{\ep -1} (a) r^{-1}_{n-1}(a)\big)\\
   & = (\al +1) a^\al e^{-t \al_\ep s^\ep_n (a)}  s_n^{\ep -1} (a)  r^{-1}_{n-1}(a)
 - \ep t\al_\ep  a^{\al +1}e^{-t \al_\ep s^\ep_n (a)}   s_n^{2\ep -2} (a) r^{-1}_{n-1}(a) (s_n(a))^{'}\\
& \quad  + (\ep -1)    a^{\al +1}e^{-t \al_\ep s^\ep_n (a)}   s_n^{\ep -2} (a) r^{-1}_{n-1}(a) (s_n(a))^{'}
            - a^{\al +1} e^{-t \al_\ep s^\ep_n (a)}  s_n^{\ep -1} (a) r^{-2}_{n-1}(a) r^{'}_{n-1}(a).
\end{align*}

Note that
\begin{align*}
 s^{'}_n(a) & =   \prod_{k=1}^{k=n-1}
 (1 + s_k)^{-1 }(a) (1+a)^{-1},\\
 r^{'}_n(a)&  = \sum_{1 \leq k \leq n}  t_k(a) \prod_{l =1}^{l =k-1}(1 + s_l)^{-1}(a) (1+a)^{-1},
 \end{align*}
  where
$t_k(a) = \frac{r_n(a)}{s_k(a)}$. Hence, we get
\begin{align*}
&\lim_{a \ri \infty} \frac{\int_1^a z^\al e^{-t \al_\ep s^\ep_n (z)} r^{-1}_{n-1}(z)  dz   }{a^{\al +1}e^{-t \al_\ep s^\ep_n (a)}  s_n^{\ep -1} (a) r^{-1}_{n-1}(a)   }\\
 &=  \lim_{a \ri \infty} \frac{1  }{(\al +1) +  \frac{a}{1 +a} \big(  -\ep t \al_\ep  s^{\ep-1}_n(a)  + (\ep -1) s_n^{-1}(a) \big)  \prod_{k=1}^{k=n-1}
 (1 + s_k(a))^{-1 }   - \frac{a}{1+a}  r^{-1}_{n-1}(a) \sum_{1\leq k \leq n} t_k(a) r^{-1}_k (a)   }\\
&  =(\al +1)^{-1}.
\end{align*}
This completes the proof of (1) of  Lemma \ref{mainlemma2}.

Using the change of variables and integration by parts sequentially, we get
\begin{align*}
\int_{a_0}^a (z+1)^{\al} e^{-t \al_\ep s^\ep_n (z)} s^{\ep-1}_n (z) r^{-1}_{n-1}(z)  dz
&= \int_{\ln \,(1 + a_0)}^{\ln \,(1 +a)} e^{(1 +\al)z} z^{-1}  e^{-t \al_\ep s^\ep_{n-1} (z)} s^{\ep-1}_{n-1} (z) r^{-1}_{n-2}(z)  dz\\
& =(1 + \al)^{-1} e^{(1 +\al)z} z^{-1} e^{-t \al_\ep s^\ep_{n-1} (z)}  s^{\ep-1}_{n-1} (z) r^{-1}_{n-2}(z)\Big|_{\ln \,(1 + a_0)}^{\ln \,(1 +a)}\\
& \qquad - (1 + \al)^{-1} \int_{\ln \,(1 + a_0)}^{\ln \,(1 +a)} e^{(1 +\al)z}  \frac{d}{dz}\Big( z^{-1} e^{-t  \al_\ep s^\ep_{n-1} (z)} s^{\ep-1}_{n-1} (z) r^{-1}_{n-2}(z) \Big) dz.
\end{align*}
Note that
\begin{align*}
&-\frac{d}{dz}\Big( z^{-1}  e^{-t \al_\ep s^\ep_{n-1} (z)} s^{\ep-1}_{n-1} (z) r^{-1}_{n-2}(z) \Big)\\
 & =  z^{-1}  e^{-t  \al_\ep s^\ep_{n-1} (z)}  s^{\ep-1}_{n-1} (z)  r^{-1}_{n-2}(z) \Big( z^{-1}  +  \ep t \al_\ep     s^{\ep-1}_{n-1}(z) \prod_{k=1}^{k=n-2}
 (1 + s_k (z))^{-1 } (1+z)^{-1}\\
 & \qquad    - (\ep -1)  s^{-1}_{n-1}(z) \prod_{k=1}^{k=n-2}
 (1 + s_k (z))^{-1 } (1+z)^{-1}  +    r^{-1}_{n-2}(z)   \sum_{1 \leq k \leq n-2}  t_k(a) \prod_{l =1}^{l =k-1}(1 + s_l (a))^{-1} (1+a)^{-1} \Big)\\
 &  =  z^{-1}  e^{-t  \al_\ep s^\ep_{n-1} (z)}  s^{\ep-1}_{n-1} (z) r^{-1}_{n-2}(z) R(z).
\end{align*}
Since $R(z)$ is a decreasing function,    for $ \ln (a_0 +1) \leq z$, we get
\begin{align*}
R(z) &  \leq   \ln (a_0 + 1)^{-1}  + \ep t \al_\ep  \ln (a_0 + 1)^{-1}  -  (\ep -1)  \ln (a_0 + 1)^{-1} +   (n-2)\ln (a_0 + 1)^{-1}\\
    & \leq (n -\ep +   \ep t \al_\ep  ) \ln (a_0 + 1)^{-1}.
\end{align*}
Since $(1 + \al)^{-1} (n -\ep +  \ep t \al_\ep  ) \ln (1 +a_0)^{-1}   = \frac12$, we have
\begin{align*}
-(1 + \al)^{-1}\frac{d}{dz}\Big( z^{-1}  e^{-t \al_\ep s^\ep_{n-1} (z)}  s^{\ep-1}_{n-1} (z)r^{-1}_{n-2}(z) \Big)
  \leq  \frac12  z^{-1}  e^{-t \al_\ep s^\ep_{n-1} (z)}  s^{\ep-1}_{n-1} (z) r^{-1}_{n-2}(z), \qquad z \geq  \ln (1 +a_0).
\end{align*}
Since $s_{n-1}( \ln \, (1 +a)) = s_n(a)  $ and $\ln\, (1 +a)    r_{n-2}(\ln \, (1 +a)) = r_{n-1}(a)  $, we have
\begin{align*}
\int_{a_0}^a (z+1)^{\al} e^{-t \al_\ep s^\ep_n (z)} s^{\ep-1} _n (z) r^{-1}_{n-1}(z)  dz
&\leq  2(1 + \al)^{-1} e^{(1 +\al)z} z^{-1}  e^{-t \al_\ep s^\ep_{n-1} (z)}  s^{\ep-1}_{n-2} (z)r^{-1}_{n-2}(z)\Big|_{{\ln \,(1 +a_0)}}^{\ln \,(1 +a)}\\
&\leq  c_\al a^{1 +\al}   e^{-t \al_\ep s^\ep_n (a)} s^{\ep-1}_n (a) r^{-1}_{n-1}(a).
\end{align*}
Hence, we complete the proof of (2) of the Lemma \ref{mainlemma2}.

For (3) of Lemma \ref{mainlemma2}, using the change of variables and integration by parts sequentially, we get
\begin{align*}
\int_a^\infty (z+1)^{\al}e^{-t \al_\ep s^\ep_n (z)} s^{\ep-1}_{n} (z) r^{-1}_{n-1}(z)  dz
&= \int_{\ln \,(1 + a)}^\infty e^{(1 +\al)z} z^{-1} e^{-t \al_\ep s^\ep_{n-1} (z)}  s^{\ep-1}_{n-1} (z) r^{-1}_{n-2}(z)  dz\\
& =(1 + \al)^{-1} e^{(1 +\al)z} z^{-1}  e^{-t  \al_\ep s^\ep_{n-1} (z)}  s^{\ep-1}_{n-1} (z) r^{-1}_{n-2}(z)\Big|_{\ln \,(1 + a)}^\infty\\
& \qquad - (1 + \al)^{-1} \int_{\ln \,(1 + a)}^\infty e^{(1 +\al)z}  \frac{d}{dz}\Big( z^{-1} e^{-t \al_\ep s^\ep_{n-1} (z)}  s^{\ep-1}_{n-1} (z) r^{-1}_{n-2}(z) \Big) dz.
\end{align*}
Since $z^{-1}  e^{-t  \al_\ep s^\ep_{n-1} (z)}  s^{\ep-1}_{n-1} (z) r^{-1}_{n-2}(z)$ is a decreasing function, $\frac{d}{dz}\Big( z^{-1}  e^{-t \al_\ep  s^\ep_{n-1} (z)}  s^{\ep-1}_{n-1} (z) r^{-1}_{n-2}(z) \Big)$ is a
non-positive function. Since $\al < -1$,  we get
\begin{align*}
\int_a^\infty (z+1)^{\al}e^{-t \al_\ep s^\ep_n (z)}   s^{\ep-1}_{n} (z)r^{-1}_{n-1}(z)  dz
&\leq
 (1 + \al)^{-1} e^{(1 +\al)z} z^{-1}  e^{-t\al_\ep s^\ep_{n-1} (z)}  s^{\ep-1}_{n-1} (z) r^{-1}_{n-2}(z)\Big|_{\ln \,(1 + a)}^\infty\\
&  = - (1 + \al)^{-1} (1 +a)^{1 + \al} e^{-t \al_\ep s^\ep_n (z)}  s^{\ep-1}_{n} (z) r^{-1}_{n-1}(a)\\
&  \leq  - (1 + \al)^{-1} a^{1 + \al} e^{-t\al_\ep  s^\ep_n (z)}  s^{\ep-1}_{n} (z) r^{-1}_{n-1}(a).
\end{align*}
Hence, we complete the proof of (3) of Lemma \ref{mainlemma2}.

\end{proof}

\section{Proof of Theorem \ref{Emaintheo}} \label{sec7}
\setcounter{equation}{0}
In this section, we prove Theorem \ref{Emaintheo}.

Let
\begin{align}\label{example1}
\notag g_1(\xi):& = \al_\ep s_n^{\ep}(\xi),\\
\notag g_2(\xi) : & =c \left\{\begin{array}{l} \vspace{2mm}
|\xi|^\ep, \quad |\xi| \leq 1,\\
s_n^{\ep-1} (\xi) \frac1{1 + r_{n-1}(\xi)}, \quad |\xi| \geq 1,
\end{array}
\right.\\
\notag g_3(\xi) : & =  c \left\{\begin{array}{l} \vspace{2mm}
   |\xi|^{\ep -1}, \quad |\xi| \leq 1,\\
 s_n^{\ep -1}(\xi) r^{-1}_{n-1}(\xi) (1 + |\xi|)^{-1}, \quad |\xi| \geq 1,
  \end{array}
  \right.\\
g_4(\xi): & = c
  \left\{\begin{array}{l}\vspace{2mm}
   |\xi|^{\ep -2}, \quad |\xi| \leq 1,\\
  s_n^{\ep -1}(\xi) r^{-1}_{n-1}(\xi)    (1 + |\xi|)^{-2}, \quad |\xi| \geq 1.
 \end{array}
 \right.
\end{align}
From the Lemma \ref{mainlemma}, it is sufficient to estimate the following;
\begin{align*}
  \frac1{|x|^2} \int_{\frac{2\pi}{|x|}}^{\infty} G(\xi) d\xi,\quad  \frac{t}{|x|}  \int_0^{\frac12 \pi}  e^{-t g_1(\frac{\xi}{x})}  g_2(\frac{\xi}{x}) \xi,\quad
 \frac1{|x|^3}      G(\frac1x),\quad      \int_0^{\frac{\pi}{2x}} G(\xi)   \xi^2 d\xi
\end{align*}
with $ G(\xi) :=  e^{-tg_1 (\xi)}\big(t^2 g_3(\xi)^2 + t g_4(\xi)   \big)$.
Note that, by the assumption  \eqref{mainassumption} of $\eta_1$, we have
\begin{align*}
e^{-t\eta_1(\xi)} \leq  e^{-tg_1(\xi)}=    e^{- \al_\ep t   s^{\ep}_n(\xi)}
\end{align*}
and
\begin{align*}
G(\xi)  \leq  c \left\{\begin{array}{l}\vspace{4mm}
e^{- \al_\ep t   s^\ep_n(\xi)}\big(t^2 |\xi|^{2\ep -2} + t |\xi|^{\ep -2}\big), \quad |\xi| \leq 1,\\
t(t+1)  e^{-t \al_\ep   s^\ep_n(\xi)}          s^{\ep-1}_n(\xi) r_{n-1}^{-1}(\xi) (1 + |\xi|)^{-2}, \quad |\xi| \geq 1.
\end{array}
\right.
\end{align*}
{\bf (1) In the case of $|x| \geq 1$.}\\
By  direct calculation, we have
\begin{align}\label{first}
\notag \frac1{|x|^2} \int_{\frac{2\pi}{|x|}}^{\infty} G(\xi) d\xi & \leq c  \frac1{|x|^2}\int_1^\infty
t(t +1) e^{-t  \al_\ep  s^\ep_n(\xi)}          s^{\ep-1}_n(\xi) r_{n-1}^{-1}(\xi) (1 + |\xi|)^{-2} d\xi\\
\notag &\quad  + \frac1{|x|^2}\int_{\frac{2\pi}{|x|}}^1    e^{-t \al_\ep  s^\ep_n(\xi)}\big(t^2 |\xi|^{2\ep -2} + t |\xi|^{\ep -2}\big)   d\xi\\
          & \leq   c t(t+1) \left\{ \begin{array}{l} \vspace{2mm}
          \frac1{|x|^{1 +\ep}}, \quad 0< \ep <1,\\
          \frac{\ln (1 + |x|)}{|x|^2}, \quad \ep =1.
          \end{array}
          \right.
\end{align}
Using the change of variables, we get
\begin{align}\label{second}
\notag \frac{t}{|x|}  \int_0^{\frac12 \pi}  e^{-t g_1(\frac{\xi}{x})}  g_2(\frac{\xi}{x}) d\xi
\notag & \leq c  t \int_0^{\frac{ \pi}{2|x|}} e^{-t g_1(\xi )}  g_2(\xi) d\xi \\
\notag & \leq c  t \int_0^{\frac{ \pi}{2|x|}} e^{- \al_\ep t    s^\ep_n(\xi)}  |\xi|^\ep d\xi \\
& \leq c t\frac1{ |x|^{1 + \ep }}.
\end{align}
Since $\frac1{|x|} \leq 1$, we have
\begin{align}\label{third}
\notag \frac1{|x|^3}      G(\frac1x)
  & \leq  c \frac1{|x|^3}    e^{- \al_\ep t    s^\ep_n(\frac1x)} \big(t^2 |\frac1x|^{2\ep -2} + t |\frac1x|^{\ep -2}\big)\\
\notag & \leq  c  t(t+1)\frac1{|x|^3}     |\frac1x|^{\ep -2}\\
& \leq  c  t(t+1)\frac1{|x|^{1 +\ep}}
\end{align}
and
\begin{align}\label{firth}
 \int_0^{\frac{\pi}{2|x|}} G(\xi)   \xi^2 d\xi & \leq  ct (t+1) \int_0^{\frac{\pi}{2|x|}}  \xi^\ep d\xi = ct (t+1) \frac1{|x|^{1 +\ep}}.
\end{align}
Hence, by \eqref{first}, \eqref{second}, \eqref{third}, \eqref{firth} and  Lemma \ref{mainlemma}, Theorem \ref{Emaintheo} holds for $|x|\geq 1$.

{\bf (1) In the case of $|x| \leq 1$.}

Note that for  $ |x| \leq 1$, taking $a = \frac{2\pi}{|x|}$ and $\al =-2$ in (3) of  Lemma \ref{mainlemma2}, we have
\begin{align}\label{first-2}
\notag \frac1{|x|^2} \int_{\frac{2\pi}{|x|}}^{\infty} G(\xi) d\xi
 & \leq  c t(t+1)  \frac1{|x|^2} \int_{\frac{2\pi}{|x|}}^{\infty}   t(t+1)  e^{- \al_\ep t    s^\ep_n(\xi)}          s^{\ep-1}_n(\xi) r_{n-1}^{-1}(\xi) (1 + |\xi|)^{-2}  d\xi\\
 \notag & \leq  c(t)   \frac1{|x|^2}    t(t+1)  e^{- \al_\ep t    s^\ep_n(\frac1x)}          s^{\ep-1}_n(\frac1x) r_{n-1}^{-1}(\frac1x)   (1 + |\frac1x|)^{-1}\\
 &  \leq c(t)  \frac1{|x|}    t(t+1)  e^{-\al_\ep t    s^\ep_n(\frac1x)}          s^{\ep-1}_n(\frac1x) r_{n-1}^{-1}(\frac1x).
\end{align}
with $c(t) \leq ct$ for $t \leq 1$. By (1) of Lemma \ref{mainlemma2}, we have
\begin{align}\label{second-2}
\notag   \int_0^{\frac{\pi}{2|x|}} G(\xi)  \xi^2 d\xi & \leq ct(t+1)   \int_0^1 |\xi|^{\ep -2}  d\xi
  +  ct(t+1)    \int_1^{\frac{\pi}{2|x|}}   e^{-t \al_\ep     s^\ep_n(\xi)}          s^{\ep-1}_n(\xi) r_{n-1}^{-1}(\xi)    d\xi \\
&\leq c(t)    \frac1{|x|}    t(t+1)  e^{- t\al_\ep     s^\ep_n(\frac1x)}          s^{\ep-1}_n(\frac1x) r_{n-1}^{-1}(\frac1x).
\end{align}
Moreover, if $t \leq 1$, then by (1) of Lemma \ref{mainlemma2}, we have
\begin{align}\label{second-2-2}
\notag   \int_0^{\frac{\pi}{2x}} G(\xi) \xi^2 d \xi &\leq ct(t+1)   \int_0^{a_0} |\xi|^{\ep} d \xi
                  +ct(t+1)   \int_{a_0}^{\frac{\pi}{2x}}   e^{- \al_\ep t    s^\ep_n(\xi)}          s^{\ep-1}_n(\xi) r_{n-1}^{-1}(\xi)   d\xi\\
&\leq c(t)    \frac1{|x|}      e^{-t   \al_\ep  s^\ep_n(\frac1x)}          s^{\ep-1}_n(\frac1x) r_{n-1}^{-1}(\frac1x),
\end{align}
where $a_0$ is a constant defined in Lemma \ref{mainlemma2}. By direct calculation, we get
\begin{align}\label{third-2}
  \frac{1}{|x|^3}G(\frac1x)
& \leq   c t(t+1)    \frac1{|x|}      e^{-t  \al_\ep   s^\ep_n(\frac1x)}          s^{\ep-1}_n(\frac1x) r_{n-1}^{-1}(\frac1x).
\end{align}
By (2) of Lemma \ref{mainlemma2}, we have
\begin{align}\label{firth-2}
\notag \frac{t}{|x|} \int_0^{\frac12 \pi}  e^{-t g_1(\frac{\xi}{x})}  g_2(\frac{\xi}{x}) d\xi
& \leq c t  \int_0^1    |\xi|^\ep    d\xi
  + c t \int_1^{\frac{\pi}{2|x|}}       e^{- \al_\ep t    s^\ep_n(\xi)}          s^{\ep-1}_n(\xi) r_{n-1}^{-1}(\xi)       d\xi \\
  &\leq c(t)       \frac1{|x|}    t(t+1)  e^{-t   \al_\ep  s^\ep_n(\frac1x)}          s^{\ep-1}_n(\frac1x) r_{n-1}^{-1}(\frac1x).
\end{align}
Hence, from \eqref{first-2} to \eqref{firth-2}, and  Lemma \ref{mainlemma}, Theorem \ref{Emaintheo} holds for $|x|\leq 1$.
$\Box$

\begin{rem}\label{rem1127}
Note that $p_t \in L^1({\mathbb R})$. In fact, using the change of variables ($ x = y^{-1}$), we get
\begin{align*}
\int_{|x| < 1}  \frac1{|x|}      e^{-\al_\ep t    s^\ep_n(\frac1x)}          s^{\ep-1}_n(\frac1x) r_{n-1}^{-1}(\frac1x) dx
&= 2 \int_0^1  \frac1{x}      e^{-\al_\ep t    s^\ep_n(\frac1x)}          s^{\ep-1}_n(\frac1x) r_{n-1}^{-1}(\frac1x) dx\\
&= 2 \int_1^\infty   \frac1{x}      e^{-\al_\ep t    s^\ep_n(x)}          s^{\ep-1}_n(x) r_{n-1}^{-1}(x) dx.
\end{align*}
Use the change of variables ($\ln \, (x +1) =y $) again, we get
\begin{align*}
&\leq c \int_{\ln \, 2}^\infty   \frac1{x}      e^{-\al_\ep t    s^\ep_{n-1}(x)}          s^{\ep-1}_{n-2}(x) r_{n-2}^{-1}(x) dx\\
& =  \cdots \\
&  \leq c \int_{s_{n-1}(1)}^\infty  \frac1{x}      e^{-\al_\ep t   (\ln \, (x +1))^\ep }          (\ln \, (x +1))^{\ep-1}_{n-2}   dx\\
& \leq c \int_{s_{n}(1)}^\infty         e^{-\al_\ep t   x^\ep }          x^{\ep-1}  dx\\
& < \infty.
\end{align*}

\end{rem}

\section{Proof of  Theorem \ref{eta2=0}} \label{sec5}
Since $\eta = \eta_1$ and  $\eta$ is symmetric, $\int_{{\mathbb R}} e^{-t\eta(\xi)} \sin \, (x\xi) d\xi =0  $. Hence,
 the inverse Fourier transform of $ f : = e^{-t\eta}$ is real and
\begin{align*}
  \int_{{\mathbb R}} e^{-t \eta(\xi) } e^{-i x\xi} d\xi
  & = \int_{{\mathbb R}} e^{-t\eta(\xi)} \cos \, (x\xi) d\xi\\
   & =\frac1{|x|} \int_{{\mathbb R}} e^{-t\eta(\frac{\xi}x)} \cos \, \xi d\xi\\
   &= \sum_{-\infty < k<\infty} \frac1{|x|} \int^{2(k+1) \pi}_{2k \pi} f(\frac{\xi}x) \cos \, \xi d\xi.
\end{align*}

Note that
\begin{align*}
f^{''}(\xi) &= e^{-t\eta(\xi)} \Big( (-t\eta^{'} (\xi) )^2    +  (-t\eta^{''} (\xi) )   \Big)\\
& \geq  e^{-t\eta(\xi)}  (-t\eta^{''} (\xi) )\\
& \geq  ct \left\{\begin{array}{l}\vspace{4mm}
  e^{- c_1 t  }  |\xi|^{\ep -2}, \quad |\xi| \leq 1,\\
  e^{-t \al_0   s^\ep_n(\xi)}          s^{\ep-1}_n(\xi) r_{n-1}^{-1}(\xi) (1 + |\xi|)^{-2}, \quad |\xi| \geq 1
\end{array}
\right.\\
& : = g(|\xi|),
\end{align*}
where $c_1$ is a  positive constant such that $g$ is decreasing.
Then, for $k \in {\mathbb Z}$, as the proof of  Lemma \ref{mainlemma}, there are $\ep_1 \in (\xi, \pi -\xi), \, \ep_2 \in (\pi + \xi, 2\pi -\xi)$ and
$\ep_3 \in (\ep_1, \ep_2)$ such that
\begin{align*}
\frac1{|x|} \int^{2(k+1) \pi}_{2k \pi} f(\frac{\xi}x) \cos \, \xi d\xi
& =  \frac{1}{|x|^3 }    \int^{\frac12 \pi}_0  (\pi -2\xi)(\ep_2 - \ep_1)f^{''} (\frac{2 \pi k + \ep_3}{x}) \cos \xi d\xi \\
& \geq \frac{1}{|x|^3 }    \int^{\frac12 \pi}_0  (\pi -2\xi)(\ep_2 - \ep_1)g (\frac{|2 \pi k + \ep_3|}{|x|}) \cos \xi d\xi\\
& \geq \frac{1}{|x|^3 }    \int_{\frac18\pi}^{\frac38 \pi} (\pi -2\xi)(\ep_2 - \ep_1)g (\frac{|2 \pi k + \frac18\pi|}{|x|}) \cos \xi d\xi\\
& \geq  c \frac{1}{|x|^3 }    g (\frac{|2 \pi k + \frac18\pi|}{|x|}).
\end{align*}
Hence, we have
\begin{align*}
 \int_{{\mathbb R}} e^{-t \eta(\xi) } e^{-i x\xi} d\xi
    &= \sum_{-\infty < k<\infty} \frac1{|x|} \int^{2(k+1) \pi}_{2k \pi} f(\frac{\xi}x) \cos \, \xi d\xi\\
    & \geq c  \sum_{0 \leq  k<\infty} \frac{1}{|x|^3 }    g (\frac{|2 \pi k + \frac18\pi|}{|x|})\\
    & \geq c  \frac{1}{|x|^2 }\int_{\frac{1}{2\pi|x|}}^\infty     g (y) dy.
\end{align*}
Hence, for $|x| \geq 1$,  we have
\begin{align*}
\int_{{\mathbb R}} e^{-t \eta(\xi) } e^{-i x\xi} d\xi
    & \geq c  \frac{t}{|x|^2 }\int_1^\infty     e^{-t \al_0   s^\ep_n(\xi)}          s^{\ep-1}_n(\xi) r_{n-1}^{-1}(\xi) |\xi|^{-2} d\xi
    + c  \frac{t}{|x|^2 } \int_{\frac1{|x|}}^1  e^{- c_1 t  }  |\xi|^{\ep -2}  d\xi \\
    & \geq c  \left\{\begin{array}{l} \vspace{2mm}
       \frac{t}{|x|^2 } +    t e^{-c_1 t} \frac{1}{|x|^{2-\ep} },\quad  0 < \ep <1,\\
        \frac{t}{|x|^2 }  (1 + e^{-c_1 t} \ln (1 + |x|)  ), \quad \ep =1.
     \end{array}
     \right.
\end{align*}
Applying (3) of Lemma \ref{mainlemma2}, for $\al =-2$ and $a = \frac1{2\pi|x|}$ for $|x| \leq 1$, we have
\begin{align*}
\int_{{\mathbb R}} e^{-t  \eta(\xi) } e^{-i x\xi} d\xi \geq c(t)       \frac1{|x|}    e^{-t  \al_0  s^\ep_n(\frac1x)}          s^{\ep-1}_n(\frac1x) r_{n-1}^{-1}(\frac1x).
\end{align*}

\section{Examples}\label{examples}
\setcounter{equation}{0}

In this section, we show that the L$\acute{\mbox{e}}$vy symbol induced from the Laplacian exponent $\psi^{\ep,n}(\xi) = (s_n (\xi))^\ep , \,\, 0 < \ep \leq 1, \,\, 1\leq n $ satisfies  assumption \eqref{mainassumption}. We also show that   $\psi^{\ep,n}$  satisfies  assumption \eqref{1001eta2=0}.

(1). We show that the L$\acute{\mbox{e}}$vy symbol induced from the Laplacian exponent $\psi^{\ep,n}$ satisfies  assumption \eqref{mainassumption}.

Clearly, $\psi^{\ep, n}$ is a Bernstein function and so $\psi^{\ep,n}$ is a Laplace exponent of some subordinators.
First, we consider in the case of $\ep =1$ and $n =1$.  Let $\psi^{1,1} := \psi^1$.  By \eqref{bernstein}, we have
\begin{align*}
\eta^1(\xi) &= \psi^1(-i\xi) =   \ln(1 -i \xi) \\
 & =  \frac12 \ln (1 + \xi^2) - iTan^{-1} \xi.
\end{align*}
Hence,   we have
\begin{align*}
\eta^1_1(\xi) & = \frac12 \ln(1 + \xi^2) ,\\
\eta^1_2(\xi) &=  Tan^{-1} \xi.
\end{align*}
It is easy to show that $\eta^1_1$ and $\eta^1_2$ satisfy the assumption in  Theorem \ref{Emaintheo} for $n =1$.
Using the mathematical induction. Suppose that $s_n(\xi)$ satisfies the assumption of Theorem \ref{Emaintheo}. Note that
\begin{align*}
\eta^{n+1} (\xi)
 = \psi^{n+1}(-i\xi) & =   \ln(1 + s_n(-i \xi) )\\
& = \ln ( 1 + \eta^n_1(\xi) - i \eta^n_2(\xi) )\\
& = \frac12 \ln  \big ( (1 + \eta^n_1(\xi))^2 + (\eta^n_2 (\xi) )^2 \big) - i Tan^{-1} \frac{\eta^n_2(\xi) }{1 + \eta^n_1(\xi)}.
\end{align*}
Hence, we get
\begin{align*}
\eta^{n+1}_1 (\xi) &= \frac12 \ln  \big ( (1 + \eta^n_1(\xi))^2 + (\eta^n_2 (\xi) )^2 \big),\\
\eta^{n+1}_2 (\xi) &=  Tan^{-1} \frac{\eta^n_2(\xi) }{1 + \eta^n_1(\xi)}.
\end{align*}
Under the assumption that $\eta^n_1$ and $\eta^n_2$ satisfies  \eqref{mainassumption}, it is easy to show that $\eta^{n+1}_1$ and $\eta^{n+1}_2$
satisfy \eqref{mainassumption}. Hence, by mathematical induction, $s_n$ satisfies \eqref{mainassumption} for all $n \geq 1$.

Next,  let $0 < \ep < 1$.
By \eqref{bernstein}, we have
\begin{align*}
\eta^{n+1}_\ep(\xi)&  = ( s_{n}(-i\xi))^\ep = \ln(1  + s_{n-1}(-i\xi))^\ep\\
& =  \ln(1  + \eta^n_{ 1}(\xi) - i \eta^n_{ 2}(\xi))^\ep \\
 & = \Big(\frac12 \ln \big( (1 +\eta^n_{ 1}(\xi) )^2 + (\eta^n_{ 2}(\xi) )^2  \big)  - iTan^{-1} \frac{\eta^n_{ 2}(\xi) }{  (1 +\eta^n_{ 1}(\xi))} \Big)^\ep \\
 & = e^{\ep \ln \Big(\frac12 \ln \big( (1 + \eta^n_{ 1}(\xi) )^2 + (\eta^n_{ 2}(\xi) )^2  \big)  - iTan^{-1} \frac{\eta^n_{ 2}(\xi) }{  (1 +\eta^n_{ 1}(\xi))} \Big)}\\
  & = e^{\frac12 \ep \ln \Big( \big(  \frac12 \ln \big( (1 + \eta^n_{ 1}(\xi))^2 + ( \eta^n_{ 2}(\xi) )^2  \big) \big)^2
                 + \big(Tan^{-1} \frac{\eta^n_{ 2}(\xi) }{  (1 +\eta^n_{ 1}(\xi))}   \big)^2 \Big)
                 - i \ep Tan^{-1} \frac{Tan^{-1} \frac{\eta^n_{ 2}(\xi) }{  (1 +\eta^n_{ 1}(\xi))} }{  \big(  \frac12 \ln \big( (1 + \eta^n_{ 1}(\xi))^2 + ( \eta^n_{ 2}(\xi))^2  \big) \big)} }.
\end{align*}
Hence,we get
\begin{align*}
\eta^{n+1}_{\ep 1}(\xi)& = \Big( \big(  \frac12 \ln \big( (1 + \eta^n_{ 1}(\xi) )^2 + (\eta^n_{ 2}(\xi))^2  \big) \big)^2
                 + \big(Tan^{-1} \frac{\eta^n_{ 2}(\xi) }{  (1 +\eta^n_{ 1}(\xi) )}   \big)^2 \Big)^{\frac12 \ep} \cos \ep Tan^{-1} \frac{Tan^{-1} \frac{\eta^n_{ 2}(\xi) }{  (1 +\eta^n_{ 1}(\xi) )} }{  \big(  \frac12 \ln \big( (1 + \eta^n_{ 1}(\xi) )^2 + (\eta^n_{ 2}(\xi) )^2  \big) \big)},\\
\eta^{n+1}_{\ep 2}(\xi) &= \Big( \big(  \frac12 \ln \big( (1 + \eta^n_{ 1}(\xi) )^2 + (\eta^n_{ 2}(\xi))^2  \big) \big)^2
                 + \big(Tan^{-1} \frac{\eta^n_{ 2}(\xi) }{  (1 +\eta^n_{ 1}(\xi) )}   \big)^2 \Big)^{\frac12 \ep} \sin  \ep Tan^{-1} \frac{Tan^{-1} \frac{\eta^n_{ 2}(\xi) }{  (1 +\eta^n_{ 1}(\xi) )} }{  \big(  \frac12 \ln \big( (1 + \eta^n_{ 1}(\xi) )^2 + (\eta^n_{ 2}(\xi))^2  \big) \big)}.
\end{align*}
Since $\eta^n_{ 1}(\xi) $ and $\eta^n_{ 2}(\xi)$ satisfy \eqref{mainassumption}, it is easy to show that $\eta^n_{\ep 1}(\xi)$ and $\eta^n_{\ep 2}(\xi)$ satisfy \eqref{mainassumption}.

(2). Second, we show that   $\psi^{\ep,n}(\xi) = (s_n (\xi))^\ep$  satisfies  assumption \eqref{1001eta2=0}.

First, we consider  the case of $\ep =1$.
By  direct calculus, we have
\begin{align*}
(\psi^{1,n}(\xi) )^{'} = s_n^{'} (\xi)  =A_{n-1}(\xi),\quad
(\psi^{1,n}(\xi) )^{''}=s_n^{''} (\xi)  = - \sum_{1\leq k \leq n-1} B_k(\xi),
\end{align*}
where for $1 \leq k \leq n-1$,
\begin{align*}
A_{n-1}(\xi): & =(1 + s_{n-1} (\xi))^{-1} (1 + s_{n-2} (\xi) )^{-1} \cdots (1 + s_1 (\xi))^{-1} (1 + \xi)^{-1},\\
 B_k(\xi):& = A_{n-1}(\xi) A_k(\xi),\quad 1 \leq k \leq n-1,\\
 B_0(\xi): &= A_{n-1}(\xi) (1 + \xi)^{-1}.
\end{align*}
Since $B_k(\xi) >0 $   for all $0 \leq k \leq n-1$, we have
\begin{align*}
-(\psi^{1,n}(\xi) )^{''}=   \sum_{1\leq k \leq n-1} B_k(\xi) > B_0(\xi)  \geq c r^{-1}_{n-1}(\xi) (1 + \xi)^{-2}.
\end{align*}

Next, we consider the case of  $0 < \ep <1$.
By direct calculus, we have
\begin{align*}
-(\psi^{\ep,n}(\xi) )^{''} & = - (\ep -1) s^{\ep-2}_n(\xi) A^2_{n-1}(\xi) + s^{\ep-1}_n(\xi) \sum_{1\leq k \leq n-1} B_k(\xi)\\
  & \geq s_{n-1}^{\ep -1}(\xi) B_0(\xi)\\
   & \geq c s_{n-1}^{\ep -1}(\xi) r^{-1}_{n-1}(\xi) (1 + \xi)^{-2}.
\end{align*}
$\Box$

\end{document}